\documentclass[reqno]{amsart}
\usepackage{amssymb}

\newtheorem{theorem}{Theorem}[section]
\newtheorem{proposition}[theorem]{Proposition}
\newtheorem{lemma}[theorem]{Lemma}

\theoremstyle{remark}
\newtheorem{remark}[theorem]{Remark}

\numberwithin{equation}{section}

\begin{document}

\title[Discrete Fourier transform and the affine Hecke algebra]
{A discrete Fourier transform associated with the affine Hecke algebra}

\author{J.F.  van Diejen}

\author{E. Emsiz}

\address{
Facultad de Matem\'aticas, Pontificia Universidad Cat\'olica de Chile,
Casilla 306, Correo 22, Santiago, Chile}
\email{diejen@mat.puc.cl, eemsiz@mat.puc.cl}


\thanks{Work was supported in part by the {\em Fondo Nacional de Desarrollo
Cient\'{\i}fico y Tecnol\'ogico (FONDECYT)} Grants \# 1090118 and  \# 11100315,
and by the {\em Anillo ACT56 `Reticulados y Simetr\'{\i}as'}
financed by the  {\em Comisi\'on Nacional de Investigaci\'on
Cient\'{\i}fica y Tecnol\'ogica (CONICYT)}}

\date{December 2011}

\begin{abstract}
We introduce an explicit representation of the double affine Hecke algebra (of type $A_1$)
at $q=1$ that gives rise to a
periodic counterpart of a well-known Fourier transform associated with
the affine Hecke algebra.
\end{abstract}

\maketitle

\section{Introduction}\label{sec1}

It is a classroom observation that any orthogonal basis of a separable Hilbert space
gives rise to a corresponding (generalized) Fourier transform (establishing an isomorphism between
the Hilbert space and an $l^2$ space over the basis indices).
Of particular interest are in this context the Fourier series corresponding to bases of classical
(basic) hypergeometric orthogonal polynomials
and their generalizations given by  Fourier-type transforms with kernels governed by
non-polynomial (basic) hypergeometric functions
\cite{koe-les-swa:hypergeometric,koe-sto:askey-wilson}.
Remarkably, such (basic) hypergeometric Fourier transforms can be fruitfully understood
as generalized spherical transforms associated with double affine Hecke algebras and their degenerations \cite{che:double,gro:fourier,koe:q-krawtchuk,mac:affine,nou-sto:askey-wilson}.

In this paper we focus on a limiting situation with the Fourier basis given by
one-dimensional Hall-Littlewood polynomials.
The corresponding Fourier transform is a well-known parameter-deformation associated with the affine
Hecke algebra (of type $A_1$)
interpolating between spherical transforms for rank-one $p$-adic symmetric spaces
\cite{car:harmonic,mau:spherical,sil:pgl2}. By extending the affine Hecke algebra to a double affine Hecke
algebra (at $q=1$), we arrive at a finite-dimensional discrete Fourier transform that may be
seen as a natural periodic analog of this deformed $p$-adic spherical transform.

The material is organized as follows. Section \ref{sec2} recalls the definition of the $A_1$-type double
affine Hecke algebra (at $q=1$) and outlines some of its key properties.
Section \ref{sec3} introduces two representations of this double affine Hecke algebra
on the space of complex functions over $\mathbb{Z}$ in terms of
difference-reflection operators and integral-reflection operators, respectively.
In Section \ref{sec4} an explicit operator intertwining the difference-reflection representation
and the integral-reflection representation is presented.
The unitarity of the difference-reflection representation with respect to an appropriate Hilbert space
structure is studied in Section \ref{sec5} and a
discrete Laplacian corresponding to a central element in the double affine Hecke algebra is introduced.
The construction of a (generalized) spherical function associated with the double affine Hecke algebra
in Section \ref{sec6} gives rise to a basis of
eigenfunctions diagonalizing the discrete Laplacian.
The eigenfunctions in question constitute the kernel of a discrete Fourier transform associated with the
affine Hecke algebra whose Plancherel formula is determined in Section \ref{sec7}.

Our discrete Laplacian and its eigenfunctions turn out to be special instances of a
quantum Hamiltonian and its Bethe-Ansatz eigenfunctions modeling an integrable
$n$-particle system discretizing the delta Bose gas on the circle \cite{die:diagonalization} (cf. Remark \ref{connections} below for further details). From this perspective, our results provide a first step towards
the construction of an appropriate Hecke-algebraic framework
in the spirit of \cite{ems-opd-sto:periodic} for the discrete quantum model introduced in Ref. \cite{die:diagonalization}.

\section{Double affine Hecke algebra}\label{sec2}

The {\em double affine Hecke algebra} $\mathbb{H}$  with parameter $\tau\in (0,1)$ (and $q=1$)
of type $A_1$
is the complex unital associative algebra  generated by invertible elements $T$, $U$ and $X$ subject
to the relations \cite[Ch. 6.1]{mac:affine}
\begin{align}\label{daha-relations}
 (T-\tau)(T+\tau^{-1})=0,\quad U^2=1,\ \text{and}\ U X U = X^{-1}=T^{-1} XT^{-1}.
\end{align}
The subalgebra $H:=\langle T,U\rangle\subset \mathbb{H}$ is referred to as the (extended)
{\em affine Hecke algebra}.

A convenient linear basis for $\mathbb{H}$ can be constructed by means of  the underlying
(extended) {\em affine Weyl group}  $W:=\langle s,u | s^2=1,u^2=1 \rangle$.  Indeed, any
group element $w\in W$ can be written uniquely in the form $w=u^{i_0}s_{i_1}\cdots s_{i_r}$,
where $s_1:=s$, $s_0:=usu$, $i_0,\ldots ,i_r\in \{ 0,1\}$, and $r\in\{ 0,1,2,\cdots \}$ is minimal.
Such a decomposition of $w$ is called {\em reduced expression} and its length $r$ defines the
{\em length} $\ell (w)$ of the group element. Upon setting $T_w:= U^{i_0} T_{i_1}\cdots T_{i_r}$,
with $T_1:=T$ and $T_0:=UTU$, one has that the algebra elements $T_wX^k$, $w\in W$, $k\in\mathbb{Z}$
form a linear basis of $\mathbb{H}$ and that the algebra elements $T_w$, $w\in W$ form a linear basis
of the subalgebra $H$. The quadratic relation for $T$ in  Eq. \eqref{daha-relations} implies the following elementary
multiplication rules in $H$
\begin{subequations}
\begin{align}
T_wT &= T_{ws}+\frac{1}{2}(1-\eta (w))(\tau-\tau^{-1})T_w,\label{TwT} \\
TT_w &= T_{sw}+\frac{1}{2}(1-\eta (w^{-1}))(\tau-\tau^{-1})T_w,\label{TTw}
\end{align}
where $\eta: W\to \{ 1,-1\}$ is defined by
\begin{equation}\label{eta-def}
\eta (w):= \ell (ws)-\ell (w).
\end{equation}
\end{subequations}

For future reference we also exhibit the multiplicative action of the generators
of $H$ on the abelian subalgebra $\langle X\rangle$ and vice versa.
In the former situation only the action of $T$ is nontrivial and given by the following classical
formula
\begin{align}
T X^k &= X^{-k} T + (\tau-\tau^{-1})\frac{X^k-X^{-k}}{1-X^{-2}}\quad (k\in\mathbb{Z})  \nonumber \\
& = X^{-k} T + \text{sign} (k) (\tau-\tau^{-1})  \sum_{j=0}^{|k|-1} X^{|k|-2j} . \label{lusztig}
\end{align}
For $k=1$ and $k=-1$ this formula becomes $TX=X^{-1}T+ (\tau-\tau^{-1})X$ and
$TX^{-1}=XT- (\tau-\tau^{-1})X$, respectively,
which are both manifest upon multiplying the relation $TX^{-1}T=X$
by $T$ (from the left or from the right, respectively) and successive elimination of the $T^2$ factor
with the aid of the quadratic relation for $T$. The general
case then readily follows by induction in $k$.

The multiplicative action of $X$ on the basis of $H$ is more intricate. To give a precise description,
let us note that the affine Weyl group decomposes as $W=\Omega \ltimes W_S$, where
$\Omega := \langle u \rangle = \{ 1, u\}$ and $W_S$ denotes the normal subgroup $\langle s_0, s_1 \rangle\subset W$.
This decomposition gives rise to the following partial
{\em Bruhat order} on $W$:
\begin{align}
w<w' \iff w^{-1}w' \in W_S    \  \text{and}\  \ell (w)<\ell (w') .
\end{align}

\begin{proposition}\label{Xaction:prp}
Let $w=u^rw^\prime$ with $r\in \{ 0 ,1\}$ and $w^\prime\in W_S$, and let $\varepsilon\in \{ -1 ,1\}$.
Then one has that
\begin{subequations}
 \begin{equation}
 T_w X^\varepsilon=X^{\varepsilon(-1)^{\ell (w)+r}} T_w +\varepsilon \eta(w) X^{\eta(w)(-1)^{\ell (w)+r}}
 \sum_{\substack{v\in W \\ v<w}} a(\ell (w)-\ell (v))T_v  ,
  \end{equation}
with $\eta (w)$ given by Eq. \eqref{eta-def} and
\begin{equation}\label{ak}
a(k):= \Bigl( \frac{1-\tau^2}{1+\tau^2}\Bigr)  (\tau^{-k}+(-1)^{k+1}\tau^k) .
\end{equation}
\end{subequations}
\end{proposition}

\begin{proof}
It is sufficient to consider the case that $r=0$, $\varepsilon =1$ and $\eta (w)=1$:
\begin{equation}\label{TwX}
T_w X = X^{(-1)^{\ell (w)}}\Bigl( T_w +\sum_{v\in W ,\, v<w} a(\ell (w)-\ell (v))T_v \Bigr)
\end{equation}
(for $w\in W_S$ with $ws>w$).
Indeed, the case $r=0$ and $r=1$ are related via the multiplication of both sides from the left by $U$, and
the formulas corresponding to the other
three combinations of signs $(\varepsilon ,\eta (w))=(-1,1),(1,-1), (-1,-1)$
follow from the case $(\varepsilon ,\eta (w))=(1,1)$ upon invoking the relations
$T_wX^{-1}=(X+X^{-1})T_w-T_wX$
(cf. Remark \ref{center:rem} below) and
$T_{uwu}X=UT_wX^{-1}U$. (Notice in this respect that $\eta(uwu)=-\eta(w)$ for $w\in W\setminus\Omega$.)
We will now proceed to prove Eq. \eqref{TwX} by induction on $\ell (w)$ starting
from the trivial case $\ell(w)=0$.
Let $w\in W_S$ with $ws>w$ and let $i=\ell (w)\mod 2$ (i.e. $i\in \{ 0,1\}$ is chosen such that
$s_iw>w$ (and
$\eta (s_iw)=1$)). Assuming that the relation in Eq. \eqref{TwX} holds for $w$, one readily deduces that
\begin{align*}
T_{s_iw}X & = T_{s_i} X^{(-1)^{\ell (w)}} \Bigl(  T_w +
 \sum_{v\in W, \, v<w} a(\ell (w)-\ell (v))T_v\Bigl) \\
&\stackrel{(i)}{=} X^{(-1)^{\ell (s_iw)}} \Bigl( T_{s_i} -(\tau-\tau^{-1})\Bigr) \Bigl(  T_w +
 \sum_{v\in W ,\, v<w} a(\ell (w)-\ell (v))T_v\Bigr) \\
 &\stackrel{(ii)}{=} X^{(-1)^{\ell (s_iw)}} \Bigl(  T_{s_iw}+
  \sum_{\substack{v\in W,\, v<s_i w\\ s_iv<v}} a(\ell (s_iw)-\ell (v))T_v \Bigr.\\
  & +\Bigl. \sum_{\substack{v\in W,\, v<s_i w\\ s_iv>v }}
  \bigl( a(\ell (s_iw)-\ell (v)-2)-(\tau-\tau^{-1})a(\ell (s_iw)-\ell (v)-1)\bigr) T_v \Bigr) \\
  &\stackrel{(iii)}{=}X^{(-1)^{\ell (s_iw)}} \Bigl(  T_{s_iw}+
  \sum_{\substack{v\in W,\, v<s_i w}} a(\ell (s_iw)-\ell (v))T_v \Bigr) .
\end{align*}
Here we have used $(i)$ the relation in
 Eq. \eqref{lusztig} with $k=-1$ (upon conjugation by
$U$ if $i=0$),
$(ii)$ the property that $T_{s_i}T_v=T_{s_iv}$ if $s_iv>v$ and
$T_{s_i}T_v=T_{s_iv}+(\tau-\tau^{-1})T_v$ if $s_iv<v$ (cf. Eq. \eqref{TTw}), and $(iii)$ the identity $a(k-2)-(\tau-\tau^{-1})a(k-1)=a(k)$.
\end{proof}

\begin{remark}\label{center:rem}
It is immediate from the relation in Eq. \eqref{lusztig} (for $k=1$ and $k=-1$) in combination with
the commutation relation
$UX^{\pm 1}=X^{\mp 1}U$ that the element $X+X^{-1}$ lies
in the  center $\mathcal{Z}(\mathbb{H})$ of the double affine Hecke algebra.
\end{remark}

\section{Representations of $\mathbb{H}$ on $\mathcal{C}(\mathbb{Z})$}\label{sec3}
For any (fixed) positive integer $M$, the action
$s n:=-n$, $u n:=M-n$ on $n\in\mathbb{Z}$ determines a faithful representation of $W$
on the space $\mathcal{C}(\mathbb{Z})$ of functions $f:\mathbb{Z}\to\mathbb{C}$ given by $(wf)(n):=f(w^{-1}n)$.
A fundamental domain with respect to the action of $W_S\subset W$ on $\mathbb{Z}$ is given by
$\Lambda_M:=\{ 0,1,2,\ldots ,M\}$. For any $n\in\mathbb{Z}$, let us denote by $w_n$ the unique shortest element in $W_S$ mapping $n$ into $\Lambda_M$, i.e.
\begin{equation}\label{wn}
w_n = \begin{cases} \cdots s_0 s_1 s_0\ \text{with}\ \ell (w_n) = p &\text{if}\ n =pM+r, \
p\in\mathbb{Z}_{>0},\ 0 < r \leq M, \\
\cdots s_1 s_0 s_1\ \text{with}\ \ell (w_n) = p &\text{if}\ n =-pM+r, \
p\in\mathbb{Z}_{>0},\ 0 \leq r < M ,\\
1 &\text{if}\ 0\leq n\leq M.
\end{cases}
\end{equation}
With the aid of this representation of $W$, we now introduce two explicit representations of the double affine Hecke algebra $\mathbb{H}$ on  $\mathcal{C}(\mathbb{Z})$.
For convenience, we often employ the shorthand conventions
$f_n$ for $f(n)$ and $n_+$ for $w_n n$, and it will moreover be assumed that $M>1$ from now on.

\subsection{Difference-reflection representation}
Let $\hat{T}:\mathcal{C}(\mathbb{Z})\to \mathcal{C}(\mathbb{Z})$ be the operator
\begin{subequations}
\begin{equation}
\hat{T}:=\tau + \tau^{\text{sign}}(s-1),
\end{equation}
where $\tau$ and $\tau^{\text{sign}}$ act by multiplication, i.e. $(\tau f)_n:=\tau f_n$ and $( \tau^{\text{sign}}f)_n:=\tau^{\text{sign}(n)} f_n$ (with the convention that $\text{sign}(0):=0$). More specifically, one has that
\begin{equation}
(\hat{T}f)_n=
\begin{cases}\tau f_{-n} & \text{if}\ n\geq 0,\\
\tau^{-1} f_{-n} +(\tau-\tau^{-1})f_n&\text{if}\ n< 0.
\end{cases}
\end{equation}
\end{subequations}
Similarly, let $\hat{X}:\mathcal{C}(\mathbb{Z})\to \mathcal{C}(\mathbb{Z})$ denote the operator characterized by
\begin{align}\label{Xhat}
(\hat{X}f)_n :=& \tau^{(\ell (w_{n-1})-\ell (w_n))(1+\eta(w_{un}))}f_{n-1}\\
&+
\text{sign}(n) \sum_{v\in W,\, v<w_n}
\tau^{-(\ell (w_{n})-\ell (v))}a(\ell (w_{n})-\ell (v)) f_{v(n-\text{sign}(n))_+},\nonumber
\end{align}
with $a(\cdot )$ taken from Eq. \eqref{ak}.

\begin{theorem}\label{DAHA-rep1:thm}
The assignment  $T\to \hat{T}$, $U\to u$, $X\to\hat{X}$ extends (uniquely) to a
representation of the double affine Hecke algebra $\mathbb{H}$ on $\mathcal{C}(\mathbb{Z})$, i.e.
\begin{align}\label{DAHA-relations1}
(\hat{T}-\tau)(\hat{T}+\tau^{-1})=0\quad
\text{ and }\quad u\hat{X} u = \hat{X}^{-1}=\hat{T}^{-1} \hat{X}\hat{T}^{-1}.
\end{align}
\end{theorem}
It is not hard to verify the quadratic relation for $\hat{T}$ in Eq. \eqref{DAHA-relations1}
directly, however, to check this way the validity of the commutation
relations involving $u$, $\hat{X}$ and $\hat{T}$, $\hat{X}$ turns out to be very cumbersome.
Instead, Theorem \ref{DAHA-rep1:thm} will follow below in Section \ref{sec4}
as a consequence of intertwining relations connecting
to a second representation of $\mathbb{H}$ in terms of integral-reflection operators.

\subsection{Integral-reflection representation}
Let $I$ be a discrete integral-reflection operator on $\mathcal{C}(\mathbb{Z})$ of the form
\begin{subequations}
\begin{equation}\label{Idef}
I:=\tau s+(\tau-\tau^{-1})J,
\end{equation}
where
\begin{align}\label{Jdef}
(Jf)_n &:=
\begin{cases}
-f_{n-2}-f_{n-4}  -\dots -f_{-n} & \text{if $n>0$}\\
0 & \text{if $n=0$}\\
f_n+f_{n+2}  +\dots +f_{-n-2} & \text{if $n<0$}
\end{cases}
\end{align}
\end{subequations}
(thus in essence integrating $f$ from $n$ to $-n$ with step $2$),
and let $D$ denote the translation operator on lattice functions given by
$(Df)_n:=f_{n-1}$.

\begin{theorem}\label{DAHA-rep2:thm}
The assignment  $T\to I$, $U\to u$, $X\to D$   extends (uniquely) to a
representation of the double affine Hecke algebra $\mathbb{H}$ on $\mathcal{C}(\mathbb{Z})$, i.e.
\begin{align}\label{DAHA-relations2}
(I-\tau)(I+\tau^{-1})=0\quad
\text{ and }\quad uD u = D^{-1}=I^{-1} D I^{-1}.
\end{align}
\end{theorem}
\begin{proof}
From the definition of $I$ it is clear that
$$ I^2-(\tau-\tau^{-1})I-1=(\tau-\tau^{-1})^2(J^2-J)+(\tau^2-1)(sJ+Js-s+1).$$
To verify the quadratic relation for $I$ it is therefore sufficient to show that
{ (a)} $J^2=J$ and {(b)} $sJ+Js=s-1$. The latter two relations follow from the following elementary observations:
$(i)$ if $f$ is odd then $Jf=f$,  $(ii)$ $Jf$ is odd for any $f$,  $(iii)$ $sf-f$ is odd for any $f$.
Indeed, one has that $J^2f \stackrel{(i), (ii)}{=} Jf$ and $(sJ+Js)f\stackrel{(ii)}{=}J(s-1)f\stackrel{(i),(iii)}{=}(s-1)f$ for any $f\in\mathcal{C}(\mathbb{Z})$.

It is immediate from the definitions that $uDu=D^{-1}$, and to infer the last relation involving $I$ and $D$ it is convenient to first recast
it in the form $D^{-1}I=I^{-1}D$. Upon remarking that $I^{-1}=I-(\tau-\tau^{-1})$ (by the quadratic relation for $I$),
substitution of Eq. \eqref{Idef},
and invoking of the elementary property $sDs=D^{-1}$, one ends up with the relation
$D=JD-D^{-1}J$, which is readily verified by acting with both sides
on an arbitrary function $f\in\mathcal{C}(\mathbb{Z})$.
\end{proof}

\begin{remark}
The relation $\hat{X}^{-1}=u\hat{X}u$ gives rise
to the following explicit formula for the action of the inverse operator
$\hat{X}^{-1}$ on $f\in\mathcal{C}(\mathbb{Z})$:
\begin{align}\label{Xhat-inv}
(\hat{X}^{-1}f)_n =& \tau^{(\ell (w_{n+1})-\ell (w_n))(1+\eta(w_{n}))}f_{n+1}\\
&-
\text{sign}(n) \sum_{v\in W,\, v<w_n}
\tau^{-(\ell (w_{n})-\ell (v))}a(\ell (w_{n})-\ell (v)) f_{v(n-\text{sign}(n))_+}.\nonumber
\end{align}
\end{remark}

\section{Intertwining operator}\label{sec4}
For
$h\in\mathbb{H}$, we will denote by $\hat{T}(h)$ and $I(h)$ the images of $h$ under the representations of  Theorem \ref{DAHA-rep1:thm} and Theorem \ref{DAHA-rep2:thm}, respectively, and we will furthermore employ the shorthand notations $\hat{T}_w:=\hat{T}(T_w)$ and $I_w:=I(T_w)$ ($w\in W$).

Let $\mathcal{J}:\mathcal{C}(\mathbb{Z})\to\mathcal{C}(\mathbb{Z})$ be the operator of the form
\begin{equation}\label{Jop}
(\mathcal{J}f)_n:= \tau^{-\ell (w_n)} (I_{w_n}f)_{n_+} \qquad (f\in \mathcal{C}(\mathbb{Z}), n\in\mathbb{Z}),
\end{equation}
where (recall) $n_+:=w_n n$.

\begin{proposition}\label{bijective:prp}
The operator $\mathcal{J}$ \eqref{Jop} constitutes a linear automorphism
of the space $\mathcal{C}(\mathbb{Z})$.
\end{proposition}

\begin{proposition}\label{intertwining:prp}
The operator $\mathcal{J}$ \eqref{Jop}
satisfies the intertwining relations
\begin{equation}
 \hat{T}\mathcal{J}=\mathcal{J}I,\quad u\mathcal{J}=\mathcal{J}u,\quad \hat{X}\mathcal{J}=\mathcal{J}D .
\end{equation}
\end{proposition}

It is immediate from Propositions \ref{bijective:prp} and \ref{intertwining:prp} that the double affine Hecke algebra relations in Theorem \ref{DAHA-rep1:thm} follow from those in Theorem \ref{DAHA-rep2:thm}. In other words, the operator
$\mathcal{J}:\mathcal{C}(\mathbb{Z})\to\mathcal{C}(\mathbb{Z})$ in Eq. \eqref{Jop}
provides an explicit intertwining operator between both representations of the double affine Hecke algebra:
\begin{equation}
\hat{T}(h)\mathcal{J}=\mathcal{J}I(h)\quad (\forall h\in\mathbb{H}).
\end{equation}

\subsection{Proof of Proposition \ref{bijective:prp}}
We need to show that $\mathcal{J}:\mathcal{C}(\mathbb{Z})\to\mathcal{C}(\mathbb{Z})$ is a bijection, which is immediate from the following triangularity property:
\begin{equation}\label{tria}
(I_{w_n}f)_{n_+}=\tau^{-\ell (w_n)}f_n+\sum_{k\in\mathbb{Z},\, k\prec n} * f_k\qquad (f\in\mathcal{C}(\mathbb{Z}),\, n\in\mathbb{Z}),
\end{equation}
with $k\prec n$ iff either (i) $w_k<w_n$ or (ii) $w_k=w_n$ and $|k|<|n|$. Here and below the star symbols
$*$ refer to the expansion coefficients of lower terms (with respect to the partial order $\prec$)
whose precise values are not relevant for the argument of the proof. Indeed, it is clear from the triangularity in Eq. \eqref{tria} that for any $g\in\mathcal{C}(\mathbb{Z})$
the linear equation $(\mathcal{J}f)_n=g_n$ ($n\in\mathbb{Z}$) can be uniquely solved
by induction in $n$ with respect to $\prec$.
We will now prove the triangularity in question by induction on $\ell (w_n)$ starting from the straightforward case
that $\ell (w_n)\leq 1$. (The case $\ell (w_n)=0$ is in fact trivial since then $n\in\Lambda_M$ and $(I_{w_n}f)_{n_+}=f_n$).)
To this end we first observe that for $\ell (w_n)=1$ the
triangularity in Eq. \eqref{tria} is a direct consequence of the following lemma.
\begin{lemma}\label{Is-tria:lem}
For $i\in \{0,1\}$, the action of $I_{s_i}$  on $f\in\mathcal{C}(\mathbb{Z})$ is of the form
\begin{equation*}\label{Isi}
(I_{s_i}f)_n=
\begin{cases}\displaystyle
\tau^{-1} f_{s_i n}+\sum_{\substack{k\in\mathbb{Z} \\ w_k<w_{s_i n}}} *f_k
+ \sum_{\substack{k\in\mathbb{Z},\, |k|<|s_i n| \\ w_k=w_{s_i n}}} *f_k &\text{if}\ \ell (w_{s_i n})>\ell (w_{n}),\\
\displaystyle\qquad\qquad\quad
\sum_{\substack{k\in\mathbb{Z} \\ w_k<w_{n}}} *f_k
+ \sum_{\substack{k\in\mathbb{Z},\, |k|\leq |n| \\ w_k=w_{n}}} *f_k &\text{if}\ \ell (w_{s_i n})\leq \ell (w_{n}) .
\end{cases}
\end{equation*}
\end{lemma}
\begin{proof}
From the formula for the
action in Eqs. \eqref{Idef}, \eqref{Jdef} and its conjugation with the action of $u$ one deduces that
\begin{equation*}
(I_{s_1}f)_n=(If)_n=
  \begin{cases}
\tau^{-1}f_{-n} - (\tau-\tau^{-1})\displaystyle\sum_{k=1}^{n-1}f_{n-2k} & \text{if $n>0$}, \\
\tau f_{-n} + (\tau-\tau^{-1})\displaystyle\sum_{k=0}^{-n-1}f_{n+2k} & \text{if $n\le 0$},
  \end{cases}
\end{equation*}
and
\begin{equation*}
(I_{s_0}f)_n=(uIuf)_n=
  \begin{cases}
\tau^{-1}f_{2M-n} - (\tau-\tau^{-1})\displaystyle\sum_{k=1}^{M-n-1}f_{n+2k} & \text{if $n<M$}, \\
\tau f_{2M-n} + (\tau-\tau^{-1})\displaystyle\sum_{k=0}^{n-M-1}f_{n-2k} & \text{if $n\ge M$} .
  \end{cases}
\end{equation*}
These explicit formulas are readily seen to imply the lemma (upon employing the formula in Eq. \eqref{wn} to determine the lengths of
the relevant group elements $w_{n-2k}$ and $w_{n+2k}$, together with the equivalences
$n>0\iff \ell(w_{s_1n})>\ell(w_n)$ and
$n<M \iff \ell(w_{s_0n})>\ell(w_n)$).
\end{proof}
Now let us assume that $n\in\mathbb{Z}$ with $\ell (w_n)\geq 1$ is such that
the triangularity in Eq. \eqref{tria} holds for all
$k\in\mathbb{Z}$ with $\ell (w_k)\leq \ell (w_n)$, and let
$m=s_i n$ with $\ell(w_m)=\ell(w_n)+1$  (so $w_m=w_ns_i$). Our induction hypothesis guarantees that
\begin{align}\label{Im}
(I_{w_m}f)_{m_+}&=(I_{w_n}I_{s_i}f)_{n_+} \\
&=
\tau^{-\ell (w_n)} (I_{s_i} f)_{n}+\sum_{\substack{l\in\mathbb{Z} \\ w_l<w_{n}}} * (I_{s_i}f)_l
+ \sum_{\substack{l\in\mathbb{Z},\, | l |<|n| \\ w_l=w_{n}}} * (I_{s_i}f)_l . \nonumber
\end{align}
Application of Lemma \ref{Is-tria:lem} to all terms on the RHS of Eq. \eqref{Im} shows that
$(I_{w_m}f)_{m_+}$ is equal to $\tau^{-\ell (w_m)}f_m$ plus a linear combination of terms involving
evaluations $f_k$ with $k\prec m$, which completes the induction step (and thus the proof of the proposition).
Indeed, one has that (i) $\tau^{-\ell (w_n)} (I_{s_i} f)_{n}=\tau^{-\ell (w_n)-1} f_{s_in}+\text{lower\, terms}=\tau^{-\ell (w_m)} f_{m}+\text{lower\, terms}$,
(ii) if $w_l<w_n$ then $w_{s_i l}\leq w_l s_i< w_ns_i=w_{s_in}=w_{m}$, and (iii) if $w_l=w_n$ with $| l | < | n |$ then $\ell (w_l s_i)=\ell (w_l)+1$ and $|s_i l |<|s_in|$, whence $w_{s_i l}=w_l s_i=w_ns_i=w_m$ and $|s_i l|< |m|$.

\subsection{Proof of Proposition \ref{intertwining:prp}}
Let $f\in\mathcal{C}(\mathbb{Z})$ and $n\in\mathbb{Z}$.

\subsubsection{Proof of the relation $\hat{T}\mathcal{J}=\mathcal{J}I$}
Straightforward manipulations reveal that
\begin{align*}
(\hat{T}\mathcal{J}f)_n&=
\tau (\mathcal{J}f)_n+\tau^{\text{sign}(n)}\bigl( (\mathcal{J}f)_{-n}-(\mathcal{J}f)_n\bigr) \\
&\stackrel{(i)}{=} \tau^{1-\ell (w_n)} (I_{w_n}f)_{n_+}+\tau^{\text{sign}(n)}\bigl( \tau^{-\ell (w_ns)} (I_{w_ns}f)_{n_+}-\tau^{-\ell (w_n)} (I_{w_n}f)_{n_+}\bigr)\\
&\stackrel{(ii)}{=} \tau^{-\ell (w_n)}\Bigl( (I_{w_ns}f)_{n_+}+\frac{1}{2}(1-\eta (w_n))(\tau-\tau^{-1})(I_{w_n}f)_{n_+}\Bigl)\\
&\stackrel{(iii)}{=}\tau^{-\ell (w_n)}(I_{w_n}If)_{n_+}=(\mathcal{J}If)_n
\end{align*}
(with $\eta:W\to \{-1,1\}$ as defined in Eq. \eqref{eta-def}).
In Steps $(i)$ and $(ii)$ we exploited that for $n\neq 0$:
$w_{-n}=w_ns$  and  $\ell (w_ns)=\ell (w_n)+\text{sign}(n)$, whereas for $n=0$:
$\tau^{-\ell (w_{n})} (I_{w_{n}}f)_{n_+}=f_0= \tau^{-1}(If)_0=\tau^{-\ell (w_ns)} (I_{w_ns}f)_{n_+}$.
Step $(iii)$ hinges in turn on the Hecke algebra relation
in Eq. \eqref{TwT}.

\subsubsection{Proof of the relation $u\mathcal{J}=\mathcal{J}u$}
It is immediate from the definitions that
\begin{align*}
(u\mathcal{J}f)_n&=(\mathcal{J}f)_{un}=\tau^{-\ell (w_{un})}(I_{w_{un}}f)_{w_{un}un} \\
& =
\tau^{-\ell (w_{n})}(uI_{w_{n}}uf)_{un_+}=\tau^{-\ell (w_{n})}(I_{w_{n}}uf)_{n_+}=(\mathcal{J}uf)_n
\end{align*}
(where it was used that $w_{un}=uw_nu$ and $I_{uwu}=uI_wu$).

\subsubsection{Proof of the relation $\hat{X}\mathcal{J}=\mathcal{J}D$}
The proof hinges on the Hecke algebra relation of Proposition \ref{Xaction:prp} (with $r=0$ and $\varepsilon =1$):
\begin{align*}
(\mathcal{J}Df)_n =&\tau^{-\ell (w_n)}(I_{w_n}Df)_{n_+}\\ =& \tau^{-\ell (w_n)} (D^{(-1)^{\ell (w_n)}}I_{w_n}f)_{n_+}\\
& +\eta(w_n) \tau^{-\ell (w_n)}
\sum_{v\in W,\, v< w_n} a(\ell (w_n)-\ell(v))(D^{\eta(w_n)(-1)^{\ell (w_n)}} I_{v}f)_{n_+} \\
\stackrel{(i)}{=}& \tau^{-\ell (w_n)} (I_{w_n}f)_{n_+ -(-1)^{\ell (w_n)}}\\
& +\text{sign}(n) \tau^{-\ell (w_n)}
\sum_{v\in W,\, v< w_n} a(\ell (w_n)-\ell(v)) (I_{v}f)_{n_+ -\text{sign}(n)(-1)^{\ell (w_n)}}  \\
\stackrel{(ii)}{=}& \tau^{(\ell (w_{n-1})-\ell (w_n))(1+\eta(w_{un}))} (\mathcal{J}f)_{n-1}\\
& +\text{sign}(n)
\sum_{v\in W,\, v< w_n} \tau^{-(\ell (w_n)-\ell (v))}a(\ell (w_n)-\ell(v)) (\mathcal{J}f)_{v^{-1}(n-\text{sign}(n))_+ }  \\
=& (\hat{X}\mathcal{J}f)_n .
\end{align*}
Here we used in Step $(i)$ that $\eta (w_n)=\text{sign}(n)$ for $n\neq 0$, and
in Step $(ii)$ that
\begin{align*}
n_+-(-1)^{\ell (w_n)}&=w_n(n-1)\\
&=\begin{cases}
s (n-1)_+ \, (=-1) &\text{if}\ n\equiv 0\mod 2M\ \ \text{with}\ n\leq 0 ,\\
s_0(n-1)_+ \, (=M+1) &\text{if}\ n\equiv M\mod 2M\ \text{with}\ n<0 , \\
(n-1)_+& \text{otherwise}\end{cases}
\end{align*}
and $n_+ -\text{sign}(n)(-1)^{\ell (w_n)}=w_{n}(n -\text{sign}(n))=(n-\text{sign}(n))_+$.
Notice in this connection that $\ell (w_{n-1})=\ell (w_{n})+1$ if $n\equiv 0\mod M$ with $n\leq 0$, and
that (for any $f\in\mathcal{C}(\mathbb{Z})$):
$f_{-1}=\tau(If)_1$,   $f_{M+1}=\tau (I_{s_0}f)_{M-1}$, and moreover
$\tau^{-\ell (v)}(I_vf)_{vn}=(\mathcal{J}f)_n$ for (any) $v\in W$ with
$vn=n_+$ (since $\tau^{-1}(If)_0=f_0$ and $\tau^{-1}(I_{s_0}f)_M=f_M$).

\section{Unitarity and the Laplacian}\label{sec5}
The assignment $T_w\to T^*_w:=T_{w^{-1}}$ ($w\in W$) extends to an antilinear anti-involution $*$ of
the affine Hecke algebra $H\subset\mathbb{H}$ turning it into a $*$-algebra.
Let $l^2(\mathbb{Z},\delta)$ be the Hilbert space inside $\mathcal{C}(\mathbb{Z})$ determined
by the inner product
\begin{equation}
\langle  f , g\rangle_\delta :=\sum_{n\in\mathbb{Z}} f_n \bar{g}_n \delta_n,\qquad \delta_n:=\tau^{2\ell (w_n)}
\end{equation}
(where the bar refers to the complex conjugate).

\begin{proposition}
The difference-reflection representation $h\to \hat{T}(h)$ ($h\in \mathbb{H}$) restricts to a unitary representation
of the affine Hecke algebra $H\subset \mathbb{H}$
into the space of bounded operators on $l^2(\mathbb{Z},\delta)$, i.e.
\begin{equation}
\langle \hat T(h)f,g\rangle_\delta=
\langle f,\hat T(h^*)g\rangle_\delta,\quad (h\in H,\, f,g\in l^2(\mathbb{Z},\delta)) .
\end{equation}
\end{proposition}

\begin{proof}
Clearly it is sufficient to prove the proposition for the generators
$u$ and $\hat{T}=\tau+\tau^{\text{sign}}(s-1)$.
The corresponding operators
are bounded in $l^2(\mathbb{Z},\delta)$ because
$\langle uf,uf\rangle_\delta=\langle f,f\rangle_\delta$ as $u\delta=\delta$ and
$\langle sf,sf\rangle_\delta=\langle \tau^{2\text{sign}} f,f\rangle_\delta$ as $s\delta=\tau^{2\text{sign}}\delta$.
(Notice in this connection also that
$\langle \tau^{\text{sign}} f,f\rangle_\delta\leq \tau^{-1}\langle f,f\rangle_\delta $.)
In an analogous manner it is seen that
$\langle uf,g\rangle_\delta = \langle f,ug\rangle_\delta$ and that
$\langle \hat{T}f,g\rangle_\delta=\langle (\tau-\tau^{\text{sign}})f +\tau^{\text{sign}}sf,g\rangle_\delta=
\langle f,  (\tau-\tau^{\text{sign}})g+\delta^{-1}s\tau^{\text{sign}}\delta g \rangle_\delta=
\langle f,\hat{T}g\rangle_\delta$ (using also that $s\tau^{\text{sign}}=\tau^{-\text{sign}}s$).
\end{proof}

Let us define the {\em Laplacian} on $ \mathcal{C}(\mathbb{Z})$ associated with the (center of the) double affine Hecke algebra
$\mathbb{H}$ as the operator
$L:\mathcal{C}(\mathbb{Z})\to \mathcal{C}(\mathbb{Z})$ of the form
\begin{equation}\label{Lop}
L:=\hat{T}(X+X^{-1})=\hat{X}+\hat{X}^{-1} .
\end{equation}
It is immediate from the action of $\hat{X}$ and $\hat{X}^{-1}$ in Eqs. \eqref{Xhat} and  \eqref{Xhat-inv},
that the Laplacian $L$ \eqref{Lop}
is a discrete difference operator whose explicit action on $f\in\mathcal{C}(\mathbb{Z})$
reads
\begin{equation}\label{Laction:eq}
(Lf)_n= \tau^{(\ell (w_{n+1})-\ell (w_n))(1+\eta(w_{n}))}f_{n+1}+
\tau^{(\ell (w_{n-1})-\ell (w_n))(1+\eta(w_{un}))}f_{n-1}
\end{equation}
($n\in\mathbb{Z}$), or equivalently:
\begin{subequations}
\begin{equation}\label{Laction:a}
(Lf)_n=a_nf_{n+1}+b_nf_{n-1}\qquad (n\in\mathbb{Z}),
\end{equation}
with
\begin{equation}\label{Laction:b}
a_n=\begin{cases} \tau^2 &\text{if}\ n\in M\mathbb{Z}_{>0}, \\ 1 & \text{otherwise} ,\end{cases}\qquad
b_n=\begin{cases} \tau^2 &\text{if}\ n\in M \mathbb{Z}_{\leq 0}, \\ 1 & \text{otherwise} .\end{cases}
\end{equation}
\end{subequations}
For $\tau\to 1$, the action of $L$ \eqref{Lop} degenerates to the action of the standard (undeformed)
discrete Laplacian on $\mathcal{C}(\mathbb{Z})$ (shifted by an additive constant  such that
the diagonal terms vanish).

\begin{proposition}\label{symmetry:prp}
The Laplacian $L$ \eqref{Lop}
constitutes a (bounded) self-adjoint operator on $l^2(\mathbb{Z},\delta)$, i.e.
\begin{equation}
\langle Lf,g\rangle_\delta =\langle f, Lg\rangle_\delta \qquad (f,g\in l^2(\mathbb{Z},\delta)).
\end{equation}
\end{proposition}

\begin{proof}
It is manifest from Eqs. \eqref{Laction:a}, \eqref{Laction:b}
that the operator $L$ is bounded on $l^2(\mathbb{Z},\delta)$.
Moreover, one has that
\begin{align*}
&\langle Lf,g\rangle_\delta=\sum_{n\in\mathbb{Z}} (a_nf_{n+1}+b_nf_{n-1})\bar{g}_n\delta_n\\
&=
\sum_{n\in\mathbb{Z}} f_n(a_{n-1}\bar{g}_{n-1}\delta_{n-1}+b_{n+1}\bar{g}_{n+1}\delta_{n+1})\\
&=
\sum_{n\in\mathbb{Z}} f_n(a_{n}\bar{g}_{n+1}+b_{n}\bar{g}_{n-1})\delta_{n}=
\langle f,Lg\rangle_\delta ,
\end{align*}
where it was used that $a_n\delta_{n}=b_{n+1}\delta_{n+1}$ ($\forall n\in\mathbb{Z}$).
\end{proof}

\section{Spherical function}\label{sec6}
It is well-known that the eigenfunctions of the (standard) discrete
Laplacian in
$\mathcal{C}(\mathbb{Z})$
restricted to the finite-dimensional invariant
subspace of $W_S$-symmetric functions are given by  the elements of the (appropriate)
discrete Fourier-cosine basis \cite{stra:discrete} (cf. also Remark \ref{tau=1} below).  For our deformed Laplacian $L$ \eqref{Lop}, the role of the cosine function will be taken over by a spherical function associated with the double affine Hecke algebra.

Parameterized by an auxiliary spectral variable $\xi\in \mathbb{R}$,  we define the {\em spherical function} in $\mathcal{C}(\mathbb{Z})$ associated with $\mathbb{H}$ as
\begin{equation}\label{spherical-function}
\Phi_\xi := \mathcal{J} \phi_\xi\quad \text{with}\quad \phi_\xi:=(1+\tau I)\mathbf{e}^{i\xi},
\end{equation}
where $\mathbf{e}^{i\xi}\in\mathcal{C}(\mathbb{Z})$ represents the plane wave function
$\mathbf{e}^{i\xi}(n):=e^{i\xi n}$, $n\in\mathbb{Z}$.
From the explicit action of $I$ \eqref{Idef}, \eqref{Jdef} on $\mathbf{e}^{i\xi}$:
$$
I\mathbf{e}^{i \xi}=\tau \mathbf{e}^{-i \xi} +
(\tau-\tau^{-1}) \frac{\mathbf{e}^{i\xi}-\mathbf{e}^{-i \xi}}{1-e^{2i \xi }} ,
$$
it is readily seen that $\phi_\xi$  can be expressed
as a linear combination of plane waves
\begin{subequations}
\begin{equation}\label{decomp1}
\phi_\xi(n) = c(\xi)e^{i\xi n}+ c(-\xi)e^{-i\xi n}  ,
\end{equation}
with
\begin{equation}\label{cf}
 c (\xi):= \frac{1-\tau^2 e^{-2i  \xi }}{1-e^{-2i\xi}} .
\end{equation}
\end{subequations}
This representation of $\phi_\xi$ can be conveniently rewritten as a linear combination of two Chebyshev polynomials of the second kind: $\phi_\xi(n)=U_n(\cos(\xi))+\tau^2 U_{-n}(\cos(\xi))$ with $U_n(\cos(\xi))=\sin ((n+1)\xi)/\sin (\xi)$, whence it is a polynomial of degree $|n|$ in  $\cos (\xi)$. These trigonometric polynomials in the spectral parameter are often referred to as the
one-dimensional Hall-Littlewood polynomials. It is manifest from the above explicit expressions that $\phi_\xi$ and thus also $\Phi_\xi$ are real-valued:
$\overline{\phi_\xi} = \phi_\xi$, $\overline{\Phi_\xi} = \Phi_\xi$.

Let $\mathcal{C}(\mathbb{Z})^{W}:=\{ f\in \mathcal{C}(\mathbb{Z})\mid wf=f,\ \forall w\in W\}$ and
$\mathcal{C}(\mathbb{Z})^{W_S}:=\{ f\in \mathcal{C}(\mathbb{Z})\mid wf=f,\ \forall w\in W_S\}$.

\begin{proposition}\label{Winvariance:prp}
The spherical function $\Phi_\xi$ \eqref{spherical-function} lies in
the subspace $\mathcal{C}(\mathbb{Z})^{W_S}$ of  $W_S$-invariant  functions if
the spectral parameter $\xi\in\mathbb{R}$ satisfies the equation
\begin{equation}\label{BAE}
e^{iM\xi} = \varepsilon \left( \frac{1-\tau^2 e^{2i\xi}}{\tau^2-e^{2i\xi}} \right) ,\qquad \varepsilon\in \{ 1,-1\} ,
\end{equation}
with the case $\varepsilon =1$ corresponding to the situation that
$\Phi_\xi\in \mathcal{C}(\mathbb{Z})^{W}\subset \mathcal{C}(\mathbb{Z})^{W_S}$.
\end{proposition}
\begin{proof}
From the quadratic relation
$T(1+\tau T)=\tau (1+\tau T)$ together with
the intertwining relations in Proposition \ref{intertwining:prp}, it is seen that
$\hat{T}\Phi_\xi=\hat{T}\mathcal{J}\phi_\xi=\mathcal{J}I\phi_\xi=\mathcal{J}I(1+\tau I)\mathbf{e}^{i\xi} =\mathcal{J}\tau (1+\tau I)\mathbf{e}^{i\xi}=\tau\Phi_\xi$,
whence $s\Phi_\xi =\Phi_\xi$ (for any $\xi\in\mathbb{R}$). Moreover, from Propositions \ref{bijective:prp}, \ref{intertwining:prp} it is clear that $u\Phi_\xi=\varepsilon \Phi_\xi$ if and only if $u\phi_\xi=\varepsilon \phi_\xi$ ($\varepsilon\in \{ 1,-1\}$).
The explicit expansion in Eq. \eqref{decomp1}, \eqref{cf}
reveals that this
relation for $\phi_\xi$ holds if $e^{iM\xi}=\varepsilon c(-\xi)/c(\xi)=\varepsilon (1-\tau^2e^{2i\xi})/(\tau^2-e^{2i\xi})$.
\end{proof}

Equation \eqref{BAE} for the spectral parameter turns out to be the simplest case of a Bethe-Ansatz equation
for the discrete Bose gas on the circle studied in Ref. \cite{die:diagonalization} (cf. Remark \ref{connections} below).
We will now describe what the solutions of these Bethe-Ansatz equations in {\em loc. cit.} amount to in the present situation (with proofs included to keep the presentation self-contained).
Let
\begin{subequations}
\begin{equation}
V(\xi ):=M\xi +\theta(\xi ),
\end{equation}
with
\begin{align}
\theta(\xi )&:=(1-\tau^4)\int_0^{2\xi }\frac{dx}{1+\tau^4-2\tau^2\cos x} \\
&=
2\arctan\bigl(\frac{1+\tau^2}{1-\tau^2}\tan \xi \bigr)
=i\log  \Bigr(  \frac{1-\tau^2 e^{2i\xi }}{e^{2i\xi }-\tau^2} \Bigr)  .
\end{align}
\end{subequations}
Here the branches of $\arctan(\cdot )$ and $\log (\cdot)$ are assumed to be chosen such that $\theta$ is quasi-periodic: $\theta(\xi +\pi)=\theta(\xi)+2\pi$ and
$\theta(\xi )$ varies from $-\pi$ to $\pi$ as $\xi$ varies from $-\pi/2$ to $\pi/2$ (which corresponds to the principal
branch).  Our parameter regime  $\tau\in (0,1)$ moreover guarantees that the odd function
 $\theta$ is smooth and strictly monotonously increasing on $\mathbb{R}$.
 The following proposition characterizes those solutions of Eq. \eqref{BAE} that are relevant for our purposes.
\begin{proposition}\label{spectrum:prp}
\begin{itemize}
\item[(i)] For any $m\in\Lambda_M$, the equation $V(\xi)=(m+1)\pi$ has a unique real-valued solution $\xi_m$.
\item[(ii)] We have that  $0< \xi_0<\xi_1<  \cdots <\xi_M <\pi$ and that $\xi_{u m}=\xi_{M-m}=\pi -\xi_m$.
\item[(iii)] The spectral value $\xi_m$ solves Eq. \eqref{BAE} with $\varepsilon =1$ if
$m\in\Lambda^e_M:=2\mathbb{Z}\cap \Lambda_M$ and
with $\varepsilon =-1$ if $m\in\Lambda^o_M:=(2\mathbb{Z}+1)\cap \Lambda_M$.
\end{itemize}
\end{proposition}
\begin{proof}
Part (i) follows from the fact that $V:\mathbb{R}\to\mathbb{R}$ is bijective. Indeed, the function
$V(\xi )$
is strictly monotonously increasing on $\mathbb{R}$ with
$ V(\xi)\to -\infty$ as $\xi\to-\infty$ and $ V(\xi)\to +\infty$ as $\xi\to+\infty$.

The first statement of Part (ii) is clear from the strict monotonicity of $V(\xi)$ and the fact that $V(0)=0$ and $V(\pi)=(M+2)\pi$; the second statement is manifest from the equality
$V(\xi_{M-m})=(M-m+1)\pi=V(\pi-\xi_m)$.

By definition, one has that $M\xi_m+\theta(\xi_m)=(m+1)\pi$. Upon multiplication by $i$ and
exponentiating both sides of the equation one deduces
that $e^{iM\xi_m}=(-1)^{m} (1-\tau^2e^{2i\xi_m})/(\tau^2-e^{2i\xi_m})$, which implies Part (iii).
\end{proof}

For the spectral values $\xi_m$, $m\in\Lambda_M$
the spherical function $\Phi_\xi$ \eqref{spherical-function} can now be written explicitly
in terms of the one-dimensional Hall-Littlewood polynomials:
\begin{equation}
\Phi_{\xi_m}(n) =c(\xi_m)e^{i \xi_m n_+}+c(-\xi_m)e^{-i\xi_m n_+}\quad (m\in\Lambda_M,\, n\in\mathbb{Z}) . \label{decomp2}
\end{equation}
Indeed,  for any $f\in\mathcal{C}(\mathbb{Z})$ and $n\in\Lambda_M$
one has that $(\mathcal{J}f)_n=f_n$  (by definition).  Hence,
 Eq. \eqref{decomp2} is immediate from Eqs. \eqref{decomp1}, \eqref{cf}, first  for $n\in\Lambda_M$ and then
(trivially) extended to arbitrary $n\in\mathbb{Z}$ using the $W_S$-invariance of
$\Phi_{\xi_m}$, $m\in\Lambda_M$ (following from Propositions \ref{Winvariance:prp} and \ref{spectrum:prp}).

\section{Discrete Fourier transform}\label{sec7}
Since $\Lambda_M$ is a fundamental domain for the action of $W_S$ on $\mathbb{Z}$, the $W_S$-invariant
subspace
$l^2(\mathbb{Z},\delta)^{W_S}:=  l^2(\mathbb{Z},\delta) \cap  \mathcal{C}(\mathbb{Z})^{W_S}$ can be
identified with the finite-dimensional space
$l^2(\Lambda_M,\Delta)$ of functions $f:\Lambda_M\to \mathbb{C}$ endowed with the inner product
\begin{subequations}
\begin{equation}\label{ip}
\langle  f , g\rangle_\Delta :=\sum_{n\in\Lambda_M} f_n \bar{g}_n \Delta_n\qquad (f,g\in l^2(\Lambda_M,\Delta)),
\end{equation}
with respect to a (suitably normalized) symmetrized weight function $\Delta:\Lambda_M\to (0,\infty)$ of the form
\begin{equation}\label{sweights}
\Delta_n :=\frac{1}{W_S(\tau^2) }  \sum_{m \in W_S n} \delta_m =
\begin{cases}
(1+\tau^2)^{-1} & \text{if $n=0,M$} ,\\
1  & \text{if $0<n<M$}.
\end{cases}
\end{equation}
\end{subequations}
Here we have employed the shorthand notation
$W_S(\tau^2):= \sum_{w\in W_S}\tau^{2\ell (w)}= 1+2\sum_{k=1}^\infty \tau^{2k}=(1+\tau^2)/(1-\tau^2)$.
The computation of $\Delta_n$ \eqref{sweights}
hinges on the following elementary identity:
\begin{align}
\sum_{m\in W_S n}\delta_m&=\sum_{m\in W_S n}\tau^{2\ell (w_m)} \nonumber \\
&=\begin{cases}
\sum_{k=0}^\infty \tau^{2k} = 1/(1-\tau^2) &\text{if}\ n=0, M,\\
(1+\tau^2)\sum_{k=0}^\infty \tau^{2k} = (1+\tau^2)/(1-\tau^2) &\text{if}\ 0< n < M.
\end{cases} \nonumber
\end{align}

Since $X+X^{-1}$ is a central element in $\mathbb{H}$ (cf. Remark \ref{center:rem} above),
it is clear that
$L\hat{T}=\hat{T}L$ and $Lu=uL$, which implies that $L$ maps the invariant subspace
$\mathcal{C}(\mathbb{Z})^{W_S}=\{ f\in \mathcal{C}(\mathbb{Z})\mid sf= f,\ uf=\pm f\}=\{ f\in \mathcal{C}(\mathbb{Z})\mid \hat{T}f=\tau f,\ uf=\pm f\}$ into itself.
Indeed, with the aid of Proposition \ref{symmetry:prp} and Eqs. \eqref{Laction:a}, \eqref{Laction:b} one infers that
$L$ restricts to a self-adjoint in operator in $l^2(\Lambda_M,\Delta)$ whose explicit action on
$f\in l^2(\Lambda_M,\Delta)$ is given by:
\begin{equation}
(Lf)_n=\begin{cases} (1+\tau^2) f_1 &\text{if}\quad n=0 ,\\
f_{n+1}+f_{n-1} &\text{if}\quad 0< n <M,\\
(1+\tau^2) f_{M-1} &\text{if}\quad n=M.
\end{cases}
\end{equation}

\begin{theorem}\label{diagonalization:thm} The spherical functions
$\Phi_{\xi_m}:\Lambda_M\to\mathbb{R}$, $m\in\Lambda_M$ are eigenfunctions of
the Laplacian $L$ in $l^2(\Lambda_M,\Delta)$:
$$ L\Phi_{\xi_m} = 2\cos(\xi_m)\Phi_{\xi_m}\qquad   (m\in\Lambda_M).$$
\end{theorem}

\begin{proof}
Elementary manipulations involving
the explicit expansion in Eq. \ref{decomp2} reveal that
$L\Phi_{\xi_m}=(\hat{X}+\hat{X}^{-1})\mathcal{J}\phi_{\xi_m}=
\mathcal{J}(D+D^{-1})(c(\xi_m)\mathbf{e}^{i\xi_m}+c(-\xi_m)\mathbf{e}^{-i\xi_m})=
2\cos(i\xi_m) \Phi_{\xi_m}$, since $(D+D^{-1})\mathbf{e}^{i\xi}=2\cos(\xi)\mathbf{e}^{i\xi}$.
To verify that $\Phi_{\xi_m}\neq 0$, it suffices to compute its value at the origin:
$\Phi_{\xi_m}(0)=c(\xi_m)+c(-\xi_m)=1+\tau^2>0$.
\end{proof}

It follows form Theorem \ref{diagonalization:thm} and Part (ii) of Proposition \ref{spectrum:prp} that the eigenvalues of $L$ in $l^2(\Lambda_M,\Delta)$ are non-degenerate. The spherical functions
$\Phi_{\xi_m}$, $m\in\Lambda_M$ thus form an orthogonal basis of
$l^2(\Lambda_M,\Delta)$.
To describe the corresponding Plancherel formula it is convenient to introduce the dual space
$l^2(\Lambda_M,\hat{\Delta})$ of functions $f:\Lambda_M\to\mathbb{C}$ endowed with the inner product
\begin{subequations}
\begin{equation}\label{dip}
\langle  f , g\rangle_{\hat{\Delta}} :=\sum_{m\in\Lambda_M} f_m \bar{g}_m \hat{\Delta}_m\qquad (f,g\in l^2(\Lambda_M,\hat{\Delta})),
\end{equation}
with respect to a weight function $\hat{\Delta}:\Lambda_M\to (0,\infty)$ of the form
\begin{align}
\hat\Delta_m &:=\frac{1}{2c(\xi_m)c(-\xi_m) V'(\xi_m)} \nonumber \\
&=
\frac{ (1-\cos(2\xi_m))}{(1+\tau^4-2\tau^2\cos(2\xi_m))}
\Bigl(M+\frac{2(1-\tau^4)}{1+\tau^4-2\tau^2\cos(2\xi_m)}\Bigr)^{-1} .
\end{align}
\end{subequations}
Let $\boldsymbol{\Phi}:\Lambda_M\times\Lambda_M\to\mathbb{R}$ denote the kernel function given by
$\boldsymbol{\Phi}_{m;n}:=\Phi_{\xi_m}(n)$ ($m,n\in\Lambda_M$).

\begin{theorem}\label{orthogonality:thm}
\begin{itemize}
\item[(i)] The functions $\boldsymbol{\Phi}_{m;\cdot}:\Lambda_M\to\mathbb{R}$, $m\in\Lambda_M$ form an orthogonal basis
of $l^2(\Lambda_M,\Delta)$, viz. for any $m,m^\prime\in\Lambda_M$
$$ \langle\boldsymbol{\Phi}_{m;\cdot},\boldsymbol{\Phi}_{m^\prime;\cdot}\rangle_\Delta= \sum_{n=0}^M  \boldsymbol{\Phi}_{m;n}\boldsymbol{\Phi}_{m^\prime ;n}\Delta_n =
\begin{cases}
0 &\text{if}\quad m\neq m^\prime ,\\
\hat{\Delta}_m^{-1}&\text{if}\quad m=m^\prime .
\end{cases}
$$
\item[(ii)] The functions $\boldsymbol{\Phi}_{\cdot ;n}:\Lambda_M\to\mathbb{R}$, $n\in\Lambda_M$ form an orthogonal basis
of $l^2(\Lambda_M,\hat{\Delta})$, viz. for any $n,n^\prime\in\Lambda_M$
$$\langle\boldsymbol{\Phi}_{\cdot;n},\boldsymbol{\Phi}_{\cdot ;n^\prime}\rangle_{\hat{\Delta}}=\sum_{m=0}^M  \boldsymbol{\Phi}_{m;n}\boldsymbol{\Phi}_{m ;n^\prime }\hat{\Delta}_m =
\begin{cases}
0 &\text{if}\quad n\neq n^\prime ,\\
\Delta_n^{-1}&\text{if}\quad n=n^\prime .
\end{cases}
$$
\end{itemize}
\end{theorem}
\begin{proof}
Part (i) and Part (ii) are equivalent, since both can be reformulated as the affirmation that the matrix
$[ \hat{\Delta}_m^{1/2}\boldsymbol{\Phi}_{m;n}\Delta_n^{1/2} ]_{0\leq m,n\leq M}$ is orthogonal.
In view of the abovementioned orthogonality
of the basis $\boldsymbol{\Phi}_{m;\cdot}=\Phi_{\xi_m}$, $m\in\Lambda_M$ in $l^2(\Lambda_M,\Delta)$, it is thus sufficient
to verify the norm formula in Part (i) for $m=m^\prime$. An explicit computation reveals that
\begin{align*}
&\sum_{n=0}^M \Phi_{\xi_m}^2 (n)\Delta_n =\sum_{n=0}^M \left( c(\xi_m)e^{i\xi_m n}+c(-\xi_m)e^{-i\xi_m n} \right)^2 \Delta_n\\
& = c^2(\xi_m) \left( \sum_{n=1}^{M-1} e^{2i\xi_m n}+\frac{1+e^{2i\xi_m M}}{1+\tau^2}\right)
+ c^2(-\xi_m) \left( \sum_{n=1}^{M-1} e^{-2i\xi_m n}+\frac{1+e^{-2i\xi_m M}}{1+\tau^2}\right) \\
&+2c(\xi_m)c(-\xi_m)\left( M+\frac{1-\tau^2}{1+\tau^2}\right) ,
\end{align*}
which is readily seen to
simplify to $2c(\xi_m )c(-\xi_m) V^\prime (\xi_m)$ upon summation of the two terminating geometric series and
subsequent elimination of all occurrences of exponentials of the form $e^{\pm 2iM\xi_m}$ via the relation
$e^{2iM\xi_m}=c^2(-\xi_m )/c^2(\xi_m)$ (cf. Propositions \ref{Winvariance:prp} and \ref{spectrum:prp}).
\end{proof}

Finally, we are now in the position to define the following {\em discrete Fourier transform} $\mathcal F:l^2(\Lambda_M,\Delta)\to l^2(\Lambda_M,\hat\Delta)$ associated with $\mathbb{H}$:
\begin{subequations}
\begin{equation}
\hat{f}_m:=(\mathcal{F}f)_m:=\langle f, \boldsymbol{\Phi}_{m;\cdot}\rangle_\Delta
=\sum_{n=0}^M f_n\boldsymbol{\Phi}_{m;n}\Delta_n\qquad (f\in l^2(\Lambda_M,\Delta), m\in \Lambda_M).
\end{equation}
According to Theorem \ref{orthogonality:thm} the transform in question constitutes
a unitary isomorphism of $l^2(\Lambda_M,\Delta)$ onto $l^2(\Lambda_M,\hat{\Delta})$
with an inversion formula given by
\begin{equation}
f_n=(\mathcal{F}^{-1}\hat{f})_n=\langle \hat{f}, \boldsymbol{\Phi}_{\cdot;n}\rangle_{\hat \Delta}
=\sum_{m=0}^M \hat{f}_m \boldsymbol{\Phi}_{m;n} \hat\Delta_m \qquad (\hat{f}\in l^2(\Lambda_M,\hat{\Delta}), n\in \Lambda_M).
\end{equation}
\end{subequations}
Furthermore, according to Theorem \ref{diagonalization:thm} the discrete Fourier transform
diagonalizes the Laplacian $L$ in $l^2(\Lambda_M,\Delta)$ in the sense
that $L= \mathcal{F}^{-1} \hat{E} \mathcal{F}$, where $\hat{E}$ denotes the real multiplication operator in $l^2(\Lambda_M,\hat{\Delta})$  given by $(\hat{E}f)_m:=2\cos (\xi_m) f_m$ ($f\in l^2(\Lambda_M,\hat{\Delta})$, $m\in\Lambda_M$).

\begin{remark}
The Hilbert space $l^2(\Lambda_M,\Delta)$ admits the following natural orthogonal decomposition:
$l^2(\Lambda_M,\Delta)=l^2(\Lambda_M,\Delta)^u \oplus l^2(\Lambda_M,\Delta)^{-u}$,
where $l^2(\Lambda_M,\Delta)^u:=\{ f\in  l^2(\Lambda_M,\Delta) \mid uf=f\}$
and $l^2(\Lambda_M,\Delta)^{-u}:=\{ f\in  l^2(\Lambda_M,\Delta) \mid uf=-f\}$ (so
$l^2(\Lambda_M,\Delta)^{u}\cong l^2(\mathbb{Z},\delta )\cap \mathcal{C}(\mathbb{Z})^W$).
Similarly, the Hilbert space $l^2(\Lambda_M,\hat{\Delta})$ decomposes orthogonally as
$l^2(\Lambda_M,\hat{\Delta})=l^2(\Lambda_M^e,\hat{\Delta})\oplus l^2(\Lambda_M^o,\hat{\Delta})$.
The discrete Fourier transform
$\mathcal{F}$ maps $l^2(\Lambda_M,\Delta)^u$ onto $l^2(\Lambda_M^e,\hat{\Delta})$
and  $l^2(\Lambda_M,\Delta)^{-u}$ onto $l^2(\Lambda_M^o,\hat{\Delta})$. Indeed, it follows
from Theorem \ref{orthogonality:thm} combined with (the proof of) Proposition \ref{Winvariance:prp} and Part (iii) of Proposition \ref{spectrum:prp}, that the spherical functions
$\Phi_{\xi_m}$, $m\in\Lambda_M^e$ and $\Phi_{\xi_m}$, $m\in\Lambda_M^o$ constitute orthogonal bases of the
subspaces $l^2(\Lambda_M,\Delta)^{u}$ and $l^2(\Lambda_M,\Delta)^{-u}$, respectively.
\end{remark}

\begin{remark}\label{tau=1} It is not difficult to infer that for $\tau\to 1$ one has that:
$ \xi_m\to \frac{m\pi}{M}$,   $\Phi_{m;n}\to 2\cos( \frac{mn\pi}{M} )$, and
$$\Delta_n,2M\hat{\Delta}_n\to
\begin{cases} 1/2 &\text{if}\ n=0,M, \\ 1 &\text{if}\ 0<n<M. \end{cases}
$$
In other words, in this limit
our discrete Fourier transform (diagonalizing the Laplacian $L$ \eqref{Lop})
degenerates as expected to the discrete cosine transform (diagonalizing the undeformed discrete Laplacian)
\cite{stra:discrete,aus-gru:fourier}.
\end{remark}

\begin{remark}\label{connections}
The orthogonality relations in Theorem \ref{orthogonality:thm} $(ii)$ are a finite counterpart of well-known orthogonality relations for the one-dimensional Hall-Littlewood polynomials $\phi_\xi(n)$, $n\in\mathbb{Z}_{\geq 0}$:
$$\frac{1}{2\pi}\int_0^\pi \phi_\xi(n)\phi_\xi(n^\prime)\frac{ \text{d}\xi}{ c(\xi )c(-\xi)} =\begin{cases} 0&\text{if}\ n\neq n^\prime ,\\ \mathcal{N}_n &\text{if}\ n=n^\prime ,\end{cases} $$
with $\mathcal{N}_0= 1+\tau^2$ and $\mathcal{N}_n=1$ for $n>0$.
They constitute the simplest instances
of (conjectural) discrete orthogonality relations for the Hall-Littlewood polynomials \cite{die:finite-dimensional}
originating from the Bethe-Ansatz method for the periodic quantum integrable particle model with pairwise
delta-potential interactions known as the delta Bose gas on the circle
\cite{die:diagonalization,lie-lin:exact,gau:fonction,kor-bog-ize:quantum}.
From this perspective, the spectral problem for our Laplacian $L$ \eqref{Lop} may be viewed as a discrete analog of the textbook spectral problem for a quantum particle on the circle in the presence of a delta-potential. Based on the results presented in this paper, it is natural to expect that the generalization of the representations in Section \ref{sec3} to
the case of double affine Hecke algebras associated with arbitrary
Weyl groups will provide an appropriate algebraic framework encoding the symmetries of the
integrable discretization of the delta Bose gas on the circle introduced in
\cite{die:diagonalization}. This would generalize (i.e. discretize) corresponding recent constructions in terms of degenerations of the double affine Hecke algebra for the
(continuous) delta Bose gas on the circle \cite{ems-opd-sto:periodic}.
It is furthermore plausible that the spectral problem for $L$ \eqref{Lop} is in turn a singular limiting case
of a transcedental spectral problem for Ruijsenaars' analytic difference version of the one-dimensional
Schr\"odinger operator with
Weierstrass $\wp$-potential \cite{rui:complete,rui:generalized,rui:relativistic}, which can be captured through a very general double affine Hecke-algebraic construction due to Cherednik \cite{che:difference-elliptic}.
From this viewpoint, the representations of the double affine Hecke algebra at $q=1$ exhibited here in rank one are expected to interpolate between the corresponding representations of the degenerate double affine Hecke algebra in \cite{ems-opd-sto:periodic} and the more general Hecke-algebraic settings of Ref. \cite{che:difference-elliptic}.
\end{remark}

\vspace{3ex}
\noindent {\bf Acknowledgments.} We thank the referees for some helpful suggestions improving the presentation.

\bibliographystyle{amsplain}

\end{document}